\documentclass[11pt]{amsart}

\usepackage{amsgen,amsthm,amstext,amsbsy,amsopn,amsfonts,amssymb, enumerate}

\usepackage{comment}
\usepackage{hyperref}
\usepackage{amsmath}
\usepackage[usenames]{color}
\usepackage{xcolor}

\newcommand\la{\langle}
\newcommand\ra{\rangle}

\renewcommand\aa{{\mathfrak a}}

\newcommand\hh{{\mathfrak h}}

\newcommand\nn{{\mathfrak n}}

\newcommand\vv{{\mathfrak v}}

\newcommand\zz{{\mathfrak z}}

\newcommand\RR{\mathbb R}
\newcommand\ZZ{\mathbb Z}

\newcommand{\lrc}{\lrcorner}

\newcommand\ad{\operatorname{ad}}
\newcommand\Ad{\operatorname{Ad}}

\newcommand\Ta{\operatorname{T}}
\newcommand\Sym{\operatorname{Sym}}

\DeclareMathOperator{\En}{E}

\theoremstyle{plain}

\newtheorem{thm}{Theorem}[section]
\newtheorem{lem}[thm]{Lemma}
\newtheorem{prop}[thm]{Proposition}
\newtheorem{cor}[thm]{Corollary}

\theoremstyle{definition}
\newtheorem{defn}[thm]{Definition}
\newtheorem{rem}[thm]{Remark}
\newtheorem{example}[thm]{Example}

\begin{document}

\title[Magnetic Killing tensors]{Magnetic Killing tensors and first integrals of the magnetic flow}

\author[Andrei Moroianu]{Andrei Moroianu}
\address{Andrei Moroianu \\ Université Paris-Saclay, CNRS,  Laboratoire de mathématiques d'Orsay, 91405, Orsay, France, 
and Institute of Mathematics “Simion Stoilow” of the Romanian Academy, 21 Calea Grivitei, 010702 Bucharest, Romania}
\email{andrei.moroianu@math.cnrs.fr}

\author[Gabriela Ovando]{Gabriela P. Ovando}
\address{Departamento de Matem\'atica, ECEN - FCEIA, Universidad Nacional de Rosario. Pellegrini 250, 2000 Rosario, Santa Fe, Argentina}
\email{gabriela@fceia.unr.edu.ar}
\thanks{
	This work was partially
	supported by a grant from the IMU-CDC and Simons Foundation. A.M. has been partially supported by the PNRR-III-C9-2023-I8 grant CF 149/31.07.2023 {\em Conformal Aspects of Geometry and Dynamics}.}





\begin{abstract}  In this work we introduce a new family of symmetric tensors generalizing Killing tensors, that we call magnetic Killing tensors. We make use of them to construct first integrals for the magnetic flow associated to a given magnetic field.  We apply the results to prove integrability of some invariant  magnetic flows (either exact or non-exact) on some 2-step nilmanifolds: the Kodaira-Thurston manifold and  Heisenberg nilmanifolds of higher dimensions.

\end{abstract}

\maketitle

\setcounter{tocdepth}{1}
\tableofcontents

\section{Introduction}\label{sec:intro}
From the mechanical perspective, the motion of
a charged particle experiencing the action of a force is described by an equation on a Riemannian manifold $(M, g)$ of the form:
$$
\nabla_{\gamma'}\gamma'= q F \gamma',$$
where $q$ represents the charge, $\nabla$ is the Levi-Civita
connection of $g$, $F$ is a skew-symmetric $(1,1)$-tensor such that the corresponding 2-form $\omega^F:=g(F\cdot, \cdot)$ is
closed, and $\gamma$ is a curve on $M$ called {\em magnetic geodesic}, or magnetic trajectory. This notion stems from electromagnetism theory. 

A particular case occurs whenever the force is trivial. In this case the solution curves are geodesics. The corresponding flow can be studied on the tangent bundle $\Ta M$ (or, equivalently, on the cotangent bundle $\Ta^*M$). Solutions to the equation above are related to integral curves of a Hamiltonian field defined on $\Ta M$.  First integrals of the geodesic flow are functions $f$ defined on $\Ta M$ which are constant along geodesics, in the sense that $t\mapsto f(\gamma'(t))$ is constant for every geodesic $\gamma$.   The first integrals of the geodesic flow which are polynomial in the momenta give rise to {\em Killing tensors} on $M$, and have been intensively studied both in the mathematics and physics literature.

In the present work we focus on the integrability of magnetic flows. There are different tools to deal with  the integrability question. For instance, in \cite{DGJ1},  Dragovi\'c, Gaji\'c and  Jovanovi\'c prove complete integrability of the magnetic flow on spheres $S^{n-1}$ for a constant homogeneous magnetic field in $\mathbb R^n$ if $n\leq 6$ and conjecture the integrability for every $n$. For $n=5,6$ the first integrals are polynomials of degree at most three. Furthermore, in \cite{DGJ2}  the authors provide a Lax representation of the equations of motion and prove complete integrability of those systems for any $n$. The integrability is provided via first integrals of degree one and two. 
In \cite{BKM},  Bolsinov,  Konaev and Matveev prove integrability of the magnetic flow on the sphere with a constant 2-form by reducing the study to the so-called degenerate Neumann system, which is known to be integrable by means of first integrals linear and quadratic in momenta. Other explicit integrable examples with quadratic in momenta first integrals can be constructed using the Maupertuis' principle, see \cite{ABM}.  In 2019, Agapov and Valyuzhenich discussed polynomial integrals of magnetic flows on the 2-torus at several energy levels \cite{AV}.

On the other side, it is known that the integrability  of the geodesic flow imposes obstructions on the topology of the supporting compact manifold. In fact,  Taimanov \cite{Ta} showed  topological obstructions for real-analytic manifolds supporting real-analytic first integrals of the geodesic flow. This does not avoid however the existence of  smoothly integrable geodesic flows on compact spaces  that do not satisfy the above obstructions, see for instance the work of Butler \cite{Bu}. In particular, most first integrals in loc. cit. are constructed from Killing vector fields. A main question that motivates this work is the following: 

\smallskip

{\em Can one define in an analogous way first integrals of a magnetic flow on $\Ta M$ by means of tensor fields on $M$?}

\smallskip

First integrals play a fundamental role in the integrability question. In the case of magnetic flows,  integrability  is  more complicated  than in the geodesic flow situation. Actually this depends on  the energy level and this is related  to the fact that a constant reparametrization of a magnetic geodesic is not a solution of the corresponding magnetic equation. The nature of the Lorentz force $F$ plays a role in the integrability question. On the other hand, in view of the topological obstructions to have real-analytic integrability, it is expected that such obstructions could occur for some magnetic flows. Indeed a natural question is:

\smallskip

{\em For which kind of magnetic fields is there  such obstruction, and which other features could affect the integrability of a given magnetic flow?}

\smallskip

Furthermore, one suspects a relationship with   the critical value of Ma\~n\'e that  makes evident different dynamical behaviors for exact 2-forms. 

\smallskip

 In the present paper we obtain the characterization of symmetric tensors on a Riemannian manifold $(M,g)$ which give rise to first integrals of the magnetic flow associated to a magnetic field. These tensors are called {\em magnetic Killing tensors}, and generalize symmetric Killing tensors to the magnetic setting. However, in contrast to Killing tensors, whose components in any fixed degree are still Killing, magnetic Killing tensors are in general sums of symmetric tensors of different degrees, whose components in any fixed degree are no longer magnetic Killing tensors. 

 We then illustrate the general theory by the study of the  magnetic flow on Lie groups equipped with a left-invariant metric. We consider left-invariant magnetic fields and construct  first integrals of the magnetic flow associated to magnetic Killing tensors. In this way we obtain several examples of compact Riemannian manifolds with non-trivial force $F$ whose corresponding magnetic flows are completely integrable. 

 Explicitly, we prove the complete integrability of  magnetic flows on Kodaira-Thurston manifolds at all energy levels whenever the Lorentz force is invariant and has rank two (and therefore does not correspond to a symplectic structure). Indeed, we start by constructing the first integrals on  $\Ta (H_3\times \mathbb R)$, which descend to compact quotients by lattices of $H_3\times \mathbb R$. The first integrals that we construct are polynomial in the momenta on the simply connected Lie group, but those descending to the compact quotients are analytical only if the invariant Lorentz force is not exact, or at low  energy levels when the Lorenz force is exact. 
 
For trivial extensions of Heisenberg groups $H_{2n+1}$ endowed with a left-invariant metric and left-invariant magnetic fields, we  construct examples of Liouville integrable magnetic flows. The Lorentz forces we consider correspond to skew symmetric derivations acting trivially on the center. As above,  we start by working at the Lie group level, and then we induce corresponding first integrals on quotients by some lattices of $H_{2n+1}\times \mathbb R^k$.

In all cases, the first integrals at the Lie group level are linear or quadratic in the momenta. Whenever we induce them to the quotient spaces, we lose the polynomial property. However, the theory exposed here seems to allow the construction of polynomial first integrals of higher degrees. To this end it would be interesting to find  magnetic Killing tensors of higher degree which are indecomposable in the sense of \cite{dBM20} or \cite[Definition 2.3]{BdBM25}, that is, which are not polynomial expressions in the metric tensor and the magnetic Killing vector fields. 

\section{Magnetic Killing tensors} \label{sec:preliminaries}

The aim of  this section is to introduce {\em magnetic Killing tensors}, which  enable the construction of first integrals for the magnetic flow on a Riemannian manifold $(M,g)$ for a given magnetic field. Later this will specified on 2-step nilpotent Lie groups and induced  compact quotients. 

We start with  the formalism introduced in \cite{HMS} for symmetric tensors. For the convenience of the reader,
we recall here the standard definitions and formulas which are relevant in the sequel.

Let $(M,g)$ denote a Riemannian manifold of dimension $n$ and let $\Ta M$ denote the tangent bundle. Elements of the $l$-fold symmetric tensor product of  $\Ta M$, denoted by $\Sym^l \Ta M \subset \Ta M^{\otimes_l}$, are linear combinations of symmetrized tensor products
$$v_1 \cdot \hdots \cdot v_l :=\displaystyle{\sum}_{\sigma \in S_l} 
v_{\sigma(1)} \otimes \hdots \otimes v_{\sigma(l)},$$
where $v_1, \hdots , v_l$ are tangent vectors in $\Ta M$.

Let $\{e_i\}$ denote a local orthonormal frame of $\Ta M$. Using the metric $g$,
we  identify $\Ta M$ with $\Ta^*M$ and thus $\Sym^2 \Ta^*
M \simeq  \Sym^2 \Ta M$. 
The scalar product $g$
induces a scalar product, also denoted by $g$, on $\Sym^p \Ta M$ defined by
$$g(v_1 \cdot \hdots\cdot v_l, w_1 \cdot \hdots \cdot w_l) := \displaystyle\sum_
{\sigma \in S_l}
 g(v_1, w_{\sigma(1)}) \cdot \hdots \cdot  g(v_l, w_{\sigma(l)}).$$
Using this scalar product, every section $K$ of $\Sym^p \Ta M$ can be identified with a polynomial
map of degree $l$ on $\Ta M$, defined by the formula $K(v_1, \hdots,  v_l) = g(K, v_1 \cdot \hdots \cdot  v_l)$. The
metric adjoint of the bundle homomorphism $v\cdot  : \Sym^l \Ta M \to \Sym^{l+1} \Ta M$, $K \to  v \cdot  K$ is the
contraction map $ v\lrc : \Sym^{l+1} \Ta M \to  \Sym^l \Ta M$, $K \to  v\lrc K$ defined by 
$$(v\, \lrc\, K)(v_1,\hdots, v_l) := K(v, v_1, \hdots , v_l).$$

Let $\nabla$ denote the Levi-Civita connection for the metric $g$ and let $\Gamma(\Sym^l\Ta M)$ denote the set of sections of $\Sym^l \Ta M$. Consider the differential operators 
$$d^s : \Gamma(\Sym^l \Ta M) \to \Gamma(\Sym^{l+1} \Ta M), \qquad K \to 
\sum e_i \cdot  \nabla_{e_i}K,$$
$$\delta : \Gamma(\Sym^{l+1} \Ta M) \to \Gamma(\Sym^l \Ta M), \qquad K \to   - \sum e_i \lrc  \nabla_{e_i}K,$$ 
where $d^s$ the (symmetric) differential and  $\delta$, the formal adjoint of $d^s$, is the divergence operator. 

The operator $d^s$ coincides with the usual diferential on functions (seen as elements of $\Gamma(\Sym^0\Ta M)$), and acts as a  derivation on symmetric products. In particular
$$d^s(fK)=df\cdot K+ f d^sK.$$
for every differentiable map $f:M \to \RR$ and section $K$ of $\Sym^l \Ta M$. 

As usual, we denote by $\chi(M)=\Gamma(\Sym^1\Ta M)$ the differentiable sections of $\Ta M$, that is the vector fields on $M$.

\begin{lem}\label{lem11} For every $K\in \Gamma(\Sym^l\Ta M)$ and for every $X\in \chi(M)$, the following relation holds:
	$$(d^sK)(\underbrace{X, \hdots,  X}_{l+1})=(l+1)(\nabla_XK)(\underbrace{X, \hdots, X}_{l}).$$
\end{lem}
\begin{proof}
As above let $\{e_i\}$ denote a orthonormal local frame. For every $X\in \chi(M)$, by applying the definition one gets
$$\begin{array}{rcl}
	(d^sK)(\underbrace{X, \hdots, X}_{l+1}) & = & \displaystyle \sum_i (e_i\cdot \nabla_{e_i}K)(\underbrace{X, \hdots, X}_{l+1})\\
	& = & (l+1)\displaystyle \sum_i g(e_i,X)\nabla_{e_i}K(X, \hdots, X) \\
	& = & (l+1)(\nabla_XK)(\underbrace{X,\hdots, X}_l).
\end{array}
$$
\end{proof}

Let $\phi:\Ta M \to \RR$ be a differentiable function on the tangent bundle. If the restriction of $\phi$ to the tangent space $\Ta_pM$ is a homogeneous polynomial of degree $l$ for every $p\in M$, then it determines a  section $K_\phi$ of $\Sym^l\Ta M$ via the formula
\begin{equation}\label{morphism}
	K_\phi(v,\hdots, v):=l!\, \phi(v),\qquad\forall v\in\Ta M.
\end{equation}
Conversely, every  section $K$ of $\Sym^l \Ta M$ determines a function $\phi_K:\Ta M \to \RR$ whose restriction to each tangent space is a homogeneous polynomial of degree $l$, by the formula
\begin{equation}\label{morphism1}
	\phi(v):= \frac{1}{l!}\,K(v,\hdots, v),\qquad\forall v\in\Ta M.
\end{equation}

\begin{lem}\label{lem1} The map  $K \longrightarrow \phi_K$ in \eqref{morphism1} is an injective homomorphism of algebras. 
	\end{lem}
	\begin{proof} Follows directly from \cite[Lemma 5.1]{CLMS}.
	\end{proof}

\begin{example}\label{Exa1}
	Let $g\in \Gamma(\Sym^2\Ta M)$ be the Riemannian metric on $M$. Then $\phi_g(v)=\frac12 v \lrc v \lrc g=\Vert v \Vert^2$.
\end{example}

\begin{defn}\label{magneticfield}
	A {\em  Lorentz force} on $(M,g)$ is a skew-symmetric endomorphism $F$ of $\Ta M$ such that associated the $2$-form
	$$\omega^F(X,Y):=g(FX,Y) \quad \mbox{for all } X,Y\in \Ta M$$
	 is closed. The 2-form $\omega^F$ is called a {\em magnetic field}. 
\end{defn}

Notice that there is a one-to-one correspondence between Lorentz forces and magnetic fields.

The Lorentz force $F$ can be extended to the algebra of symmetric tensors as a $C^\infty(M)$-linear derivation $F_*$. For every $l\geq 1$ and $K_l\in \Gamma(\Sym^l\Ta M)$, one has
$$(F_*K_l)(v, \hdots ,v)= - l K_l(Fv, v, \hdots , v)\quad \mbox{ for all } v\in \chi (M).$$ 

Fix a Lorentz force on the Riemannian manifold $(M,g)$. A differential curve $\gamma:(-\varepsilon, \varepsilon)\to M$ is called a {\em magnetic trajectory or magnetic geodesic} whenever
\begin{equation}\label{magneticc}
	\nabla_{\gamma'}\gamma'=F\gamma', 
\end{equation}
where we denote by $\gamma':=\gamma'(t)$ for simplicity.

\begin{defn}
	A differentiable function $\Psi:\Ta M \to \RR$  is called a {\em first integral} of the magnetic flow for the Lorentz force $F$ if $\Psi(\gamma')$ is constant for every magnetic trajectory $\gamma$.
\end{defn}

\begin{example} The function $\phi_g$ of Example \ref{Exa1} is a first integral of the magnetic flow for every Lorentz force $F$.  Indeed, for every magnetic geodesic $\gamma$, one has
	$$\frac{d}{dt}\phi_g(\gamma')=\frac12\frac{d}{dt}g(\gamma',\gamma')=g(\nabla_{\gamma'}\gamma', \gamma')=g(F\gamma', \gamma')=0.$$
	This first integral $\phi_g$ is usually called the {\em energy function} and is denoted by $\En$.
	\end{example}
\begin{prop} \label{prop1} Let $F$ denote a Lorentz force on $(M,g)$ and for $i=0, \hdots, l$ let $K_i\in \Gamma(\Sym^i\Ta M)$ be symmetric tensors with associated functions $\phi_i(v):=\frac{1}{i!}K_i(v, \hdots, v)$ as in \eqref{morphism1}. Then 
	$\phi:=\displaystyle \sum_{i=0}^l \phi_i$
	is a first integral of the magnetic flow if and only if 
	\begin{equation}\label{system}d^sK_i=F_* K_{i+1} \quad \mbox{ for all } i= 0, \hdots, l,
	\end{equation}
	with the convention $K_{l+1}\equiv 0$.
	\end{prop}
	\begin{proof} Let $\gamma$ be a magnetic trajectory for $F$. By applying properties and  Lemma \ref{lem11}, one has
		$$
		\begin{array}{rcl}
			\gamma'(\phi(\gamma')) & = & \displaystyle \sum_{i=0}^l \gamma'(\phi_i(\gamma'))\\
			& = & \displaystyle \sum_{i=0}^l \frac{1}{i!}\gamma'(K_i(\gamma', \hdots, \gamma'))\\
		& = & \displaystyle \sum_{i=0}^l \frac{1}{i!} (\nabla_{\gamma'}K_i)(\gamma', \hdots, \gamma') + \displaystyle \sum_{i=1}^l \frac{1}{(i-1)!}K_i(F\gamma', \gamma', \hdots, \gamma')\\
		& = & \displaystyle \sum_{i=0}^l \frac{1}{(i+1)!} d^sK_i(\underbrace{\gamma', \hdots, \gamma'}_{i+1}) - \displaystyle \sum_{i=1}^l \frac{1}{i!}(F_*K_i)(\underbrace{\gamma', \hdots, \gamma'}_i)\\
	& = & 	\displaystyle \sum_{i=0}^l \frac{1}{(i+1)!} d^sK_i(\underbrace{\gamma', \hdots, \gamma'}_{i+1}) - \displaystyle \sum_{i=0}^{l-1} \frac{1}{(i+1)!}(F_*K_{i+1})(\underbrace{\gamma', \hdots, \gamma'}_{i+1})\\
	& = & \displaystyle \sum_{i=0}^l \frac{1}{(i+1)!} \left( d^sK_i-F_*K_{i+1}\right) (\underbrace{\gamma', \hdots, \gamma'}_{i+1}).
	\end{array}
$$
So, $\phi$ is a first integral of the magnetic flow if and only if the system \eqref{system} holds. 
	\end{proof}

\begin{defn}
	A symmetric tensor $K=\sum_{i=0}^l K_i$ with $K_i\in \Gamma(\Sym^i\Ta M)$ satisfying the system \eqref{system} is called a {\em magnetic Killing tensor}. 
\end{defn}

By the above proposition, magnetic Killing tensors correspond to first integrals of the magnetic flow which are polynomial in the momenta.

\begin{rem} Assume $\phi$  as above has the form $\phi=\phi_1+\phi_0$ is a first integral for $F$, so that  $\phi_1=:\xi$ is a vector field  on $M$ and $\phi_0=:f$ is a function on $M$. From the proposition above, $\phi$ is a first integral of the magnetic flow if and only if
	$d^s \xi\equiv 0,\, df= F(\xi)^\flat.$ Here $F(\xi)^\flat\in \Ta^*M$ is the metric dual of $F(\xi)$, defined by $F(\xi)^\flat(X)=g(F(\xi),X)$ for every $X\in \Ta M$. 
	
By Lemma \ref{lem11}, $d^s\xi =0$ is equivalent to, $g(\nabla_X\xi,X)=0$ for every $X$, i.e to the fact that $\xi$ is a Killing vector field on $M$. 
	\end{rem}
\begin{defn}\label{magneticKilling}
A {\em  magnetic Killing} vector field on $M$ for the Lorentz force  $F$ is a tensor $\xi^F=\xi+f$ where $\xi\in \chi(M)$ is a Killing vector field  and $f:M \to \RR$ satisfies $df= F(\xi)^\flat$.
\end{defn}

By Proposition \ref{prop1} (applied to $l=1$), every magnetic Killing vector field is a first integral of the magnetic flow with Lorenz force $F$.

\begin{rem}
The existence of magnetic Killing vector fields is not guaranteed by the existence of Killing vector fields. We shall see examples of this in the next sections.
\end{rem}
\section{The magnetic flow on Lie groups}

Let $N$ denote a  Lie group with Lie algebra  $\nn$.
Thus,  $\nn$  is the Lie algebra of left-invariant vector fields, that is, satisfying ${L_h}_*(v)=v\circ L_h$, where $L_h:N \to N$ denotes the translation on the left by $h\in N$. The Lie algebra can be identified with $\nn\simeq \Ta_eN$, where $e$ is the identity element of $N$.
For every $p\in N$, any vector in $\Ta_p N$ can be written as $pv:={L_p}_*(v)$ for some $v\in \nn\simeq \Ta_e N$. 

Besides left-translations, other diffeomorphisms of $N$ are given by right-translations and conjugation maps, $R_h, I_h:N \to N$ respectively
$$R_h(p)=ph \qquad I_h(p)=h p h^{-1}, \qquad\forall p\in N.$$
The corresponding differentials at $e$ are denoted by ${R_h}_*:\nn\to\Ta_hN$ and $Ad(h):\nn \to \nn$. The latter, called the adjoint map, is a Lie algebra isomorphism.

Analogously to left-invariant vector fields, one has right-invariant vector fields which satisfy ${R_h}_*(v)=v \circ R_h$ for all $h\in N$, and which are also determined by their value at $e\in N$.  
In terms of left-invariant vector fields, every right-invariant vector field defined by some $v\in \Ta_eN\simeq \nn$ can be written 
\begin{equation}\label{barv}
    \bar{v}_p=p (Ad(p^{-1})v),\qquad\forall p\in N.
    \end{equation}

One has the relation $\Ad(\exp X)=e^{\ad(X)}$ where $\ad(X):\nn \to \nn$ denotes the derivation given by $\ad(X)(Y)=[X,Y]$ for $X,Y\in \nn$. 

Recall that the Lie bracket of  differentiable vector fields $X,Y$ is defined  in terms of the associated flows $\varphi^X$, $\varphi^Y$  by
$$[X,Y]_p=\left.\frac{d}{dt}\right|_{t=0} \left.\frac{d}{ds}\right|_{s=0}
\varphi^X_{-t}\circ \varphi^Y_s\circ \varphi^X_t(p).$$

To every $X\in \nn$ we associate the ``horizontal'' and ``vertical'' lifts $\tilde{X}, X^*$ which are vector fields on $\Ta N$ defined respectively by the flows $\psi_t^X$ and $\varphi_t^X$, given by
$$\quad \psi_t^X(pV):=p\exp(tX) V \mbox{ for }\widetilde{X},\qquad \mbox{and} \quad \varphi_t^X(pV):=p(V+tX)\mbox{ for } {X^*},$$
for every $p\in N$ and $V\in \nn$.
These vector fields satisfy \begin{equation}\label{Verthor}\pi_*\widetilde{X}=X, \qquad \quad  X^*_V=X_{\pi(V)}, \quad \pi_*(X^*)=0,
	\end{equation}
and the following Lie bracket relations:
\begin{equation}\label{brackets}
    [\widetilde{X}, \widetilde{Y}]_{pV}=\widetilde{[X,Y]}_{pV}\quad [\widetilde{X}, {Y}^*]_{pV}=0 \qquad [{X}^*, {Y}^*]_{pV}=0.\end{equation}
Indeed, 

\begin{eqnarray*}
[\widetilde{X}, \widetilde{Y}]_{pV} & = & \left.\frac{d}{dt}\right|_{t=0} \left.\frac{d}{ds}\right|_{s=0}
I_{\exp (tX)}(\exp (sX) V)\\
& = & \left.\frac{d}{dt}\right|_{t=0} p Ad(\exp( tX))(Y) V = \left.\frac{d}{dt}\right|_{t=0} p \exp(t\ad(X)(Y)) V\\
& = & \widetilde{[X,Y]}_{pV};\\
\\
{	[\tilde{X}, {Y}^*]_{pV}} & = & \left.\frac{d}{dt}\right|_{t=0} \left.\frac{d}{ds}\right|_{s=0}
	\psi^X_{-t}\circ \varphi^Y_s \circ \psi^X_{t}(pV)\\
	& = & \left.\frac{d}{dt}\right|_{t=0} p \exp(tX)\exp(-tX)(sY + V) = \left.\frac{d}{dt}\right|_{t=0} p(V+sY)\\
	& = & 0.
\end{eqnarray*}

Analogously one obtains $[X^*,Y^*]_{pV}=0$ for every $p\in N, V\in \nn$. 

Clearly every vector  $U\in \Ta(\Ta N)$ at $pV$ decomposes uniquely as a direct sum $\widetilde{X}+Y^*$. In fact, by denoting $\pi:\Ta N \to N$, one has
$$\pi_*(U_{pV})= pX\quad \mbox{ for } X\in \nn.$$

Take $Y^*_{pV}=U_{pV}-\widetilde{X}_{pV}$, so we have
$$\Ta_{pV}(\Ta N)=\nn \oplus \ker \pi_*,$$
where we identify $\nn\simeq \Ta_pN=\pi_*\Ta_{pV}(\Ta N)$.

\begin{defn} A symmetric tensor on the Lie group  $N$ is said {\em left-invariant} if it is invariant by translations on the left by elements $h\in N$: $K_p=(L_p)_*K_e$ for every $p\in N$. 
	
	In particular a left-invariant metric on $N$ is a Riemannian metric on $N$ for which translations on the left of the form $L_h$ by elements $h\in N$ are isometries. 
	
	 Fix such metric $\la\,,\,\ra$, whose values are determined by its value at the Lie algebra  $\nn\simeq \Ta_eN$ also denoted by $\la\,,\,\ra$.
	\end{defn} 

The tautological 1-form on $\Ta N$ is given by
\begin{equation}\label{eta}\eta_{pV}(\widetilde{X})=\la X, V\ra, \qquad \eta_{pV}({X}^*)= 0, \quad \forall  p\in N,\ \forall X,V\in \nn.
	\end{equation}
By \eqref{brackets} we easily get for every $p\in N$ and $X,Y,V\in\nn$:
\begin{equation}\label{deta}d\eta_{pV}(\widetilde{X},\widetilde{Y})=-\la V, [X,Y]\ra,\qquad d\eta_{pV}({X}^*,{Y}^*)=0, \qquad d\eta_{pV}(\widetilde{X},{Y}^*)= -\la X, Y\ra.
	\end{equation}

The 2-form $\Omega := - d\eta$ is thus a symplectic form on $\Ta N$. 

A Lorentz force $F$ is called {\em left-invariant} if it is invariant by translations on the left. From now on, we work with left-invariant Lorentz forces. 

To study the magnetic flow we consider the symplectic form defined by
$$\Omega^F:=\Omega + \pi^* \omega^F.$$
By \eqref{deta} one has
$$\begin{array}{rcl}
	\Omega^F(\widetilde{X}, \widetilde{Y}) & = & \la V, [X,Y]\ra + \omega^F(X,Y)\\
		\Omega^F(\widetilde{X}, {Y}^*)  & = & \la X,Y\ra\\
\Omega^F({X}^*, {Y}^*)  & = & 0.
\end{array}
$$ 
\begin{defn}
	The Hamiltonian vector field $X_{\phi}$ in $\Ta(\Ta N)$ associated to a differentiable function $\phi: \Ta N \to \RR$ is defined by $X_{\phi}\lrc \, \Omega^F=d\phi$, or, equivalently, by
	\begin{equation}
		\Omega^F(X_{\phi},U)=U(\phi)\quad \forall U\in \chi(\Ta N)
	\end{equation}

The Poisson bracket of functions $\phi_1, \phi_2:\Ta N\to \RR$ is the function $\{\phi_1, \phi_2\}^F:\Ta N\to \RR$ defined by
$$\{\phi_1, \phi_2\}^F(pV):=\Omega^F(X_{\phi_1}, X_{\phi_2})\quad \forall p\in N,\ \forall V\in \nn.$$ 
If $F=0$, we simply denote the Poisson bracket of $\phi_1, \phi_2$ by $\{\phi_1, \phi_2\}$.
\end{defn}

The Poisson bracket is a Lie bracket on $C^{\infty}(\Ta N)$ and it also satisfies the Leibniz's rule:

$\{\phi_1, \phi_2 \phi_3\}^F=\phi_3\{\phi_1,\phi_2\}^F+\phi_2 \{\phi_1,\phi_3\}^F.$

\begin{lem} \label{l33} Let $K\in \Gamma(\Sym^l\Ta N)$ be a left-invariant symmetric tensor and let $\phi(V):=\frac{1}{l!}K(V, \hdots, V)$ be the corresponding function on $\Ta N$ as in Lemma \ref{lem1}. Then for every $X\in \nn$ one has
	\begin{itemize}
		\item $\widetilde{X}(\phi)=0$;
		\item $X^*(\phi)(pV)=\phi_{X\lrc\, K}=\frac{1}{(l-1)!}K(X,V, \hdots, V)$ for all $l\geq 1$, and $X^*(\phi)(pV)=0$ for $l=0$.
	\end{itemize}
\end{lem}
\begin{proof}
By using the definitions we get
\begin{eqnarray*}
\widetilde{X}(\phi)(pV) & = & \left.\frac{d}{dt}\right|_{t=0} \phi(p \exp(tX) V) =\frac{1}{l!} \left.\frac{d}{dt}\right|_{t=0} K_{p\exp (tX)}K(p\exp (tX) V, \hdots, p\exp (tX) V)\\
& = & \frac{1}{l!} K_{e}K(V, \hdots, V),\\
\\
{{X}^*(\phi)(pV)} & = & \left.\frac{d}{dt}\right|_{t=0} \phi(p(V+tX)) = \frac{1}{l!} \left.\frac{d}{dt}\right|_{t=0} K_{e}K(V+tX, \hdots, V + tX)\\
& = & \left\{ {\begin{array}{ll}
	{\frac{1}{(l-1)!}K(X,V,\hdots, V)} & l\geq 1\\
{0} & l=0 	
\end{array} }\right.
\end{eqnarray*}

\end{proof}

Let $\{e_i\}$ denote an orthonormal basis on $\nn$, and consider the associated vertical and horizontal vector fields $e^*_i, \widetilde{e}_i$ on $\Ta N$. These vector fields form a basis of $\Ta_{pV}(\Ta N)$ for every $p\in N$ and $V\in \nn$.

For a left-invariant symmetric tensor $K\in \Gamma(\Sym^{k+1}\Ta N)$ with associated function $\phi=\phi_K$ as in \eqref{morphism1}, we decompose
$$(X_{\phi})_{pV}=\sum \alpha_i(pV) \,e^*_i  +  \beta_i(pV)\,\widetilde{e}_i,$$
for some smooth functions $\alpha_i$ and $\beta_i$ on $\Ta N$. 
Lemma \ref{l33} implies that for every $j$ one has
$$0=\widetilde{e}_j(\phi)=\Omega^F(X_{\phi}, \widetilde{e}_j)=-\alpha_j(pV) + \sum \beta_i(pV)\left( \la V, [e_i,e_j]\ra + \omega^F(e_i,e_j)\right),$$
$$\frac{1}{k!}K(e_j,V, \hdots, V)=e_j^*(\phi)=\Omega^F(X_{\phi, e_j^*})=\beta_j(pV).$$

These relations give the expression of $(X_{\phi})$ with respect to the above basis of $\Ta(\Ta N)$ at every point $pV\in \Ta N$ (with $p\in N$ and $V\in\nn$):
$$(X_{\phi})_{pV}=\sum_j \frac{1}{k!} K(e_j, V, \hdots, V)\,\tilde{e}_j  + \sum_{i,j}(\frac{1}{k!} K(e_i,V, \hdots, V))(\la V, [e_i,e_j]+\omega^F(e_i,e_j)) \,e^*_j.$$

By using these formulas and the properties of $\Omega^F$ given above one obtains the following:

\begin{cor} Let $K\in \Gamma(\Sym^{k+1}\Ta N)$, $L\in\Gamma(\Sym^{l+1}\Ta N)$ be left-invariant symmetric tensors. Then the Poisson bracket of the associated functions $\phi_K, \phi_L$ satisfies
	\begin{eqnarray*}
	\{\phi_L,\phi_K\}^F(pV) & = & \Omega^F(X_{\phi_K}, X_{\phi_L}) \\
	& = & \sum_{i,j}\frac{1}{k!}\frac1{l!}L(e_j, V, \hdots, V) K(e_i,V\hdots, V)(\la V,[e_i,e_j]+\omega^F(e_i,e_j)).
\end{eqnarray*}
\end{cor}

With the same assumptions, the formula  above can be read in the following way:
\begin{equation}\label{poissonKL}
	\begin{split}
	\{\phi_L, \phi_K\}^F & = \sum \phi_{e_j\lrc L} \cdot \phi_{e_i\lrc K}(\la V, [e_i,e_j]+\omega^F(e_i,e_j))\\
	& =  \phi_{\{L,K\}^F},
	\end{split}
\end{equation}
where by definition 
$$\{L,K\}^F:=\sum_{i,j} (e_i\lrc L) \cdot (e_j \lrc K) (\la V, [e_i,e_j]\ra + \omega^F(e_i,e_j)).$$

For $f\in C^{\infty}(N)$, consider  $\phi_f=f\circ \pi$. Then
$$X^*(\phi_f)=0 \qquad \widetilde{X}(\phi_f)=p\exp X (f\circ \pi).$$
We now compute the Hamiltonian vector field for $\phi_f$. As before write
$$X_{\phi_f}=\sum_i \alpha_i\,e_i^*  +\beta_i\, \widetilde{e}_i ,$$
and compute
$$0=e^*_j(\phi_f)=\Omega^F(X^F_{\phi_f}, e^*_j)=\sum_i \beta_i \delta_{ij}=\beta_j$$
$$(pe_j)(f) = \widetilde{e}_j(\phi_f)(pV)=\Omega^F(X_{\phi_f}, \widetilde{e}_j)=-\alpha_j(pV).$$
Thus, the Hamiltonian vector field corresponding to $\phi_f$ is given by
$$(X_{\phi_f})_{pV}= -\sum_i (pe_i)(f)\, e_i^*, $$
where $(pe_i)(f)=\left.\frac{d}{dt}\right|_{t=0} f(p \exp(te_i))$. This, together with Lemma \ref{l33}, implies that for $K\in \Gamma(\Sym^{k}\Ta N)$ left-invariant and $f\in C^{\infty}(N)$ as above, one has
$$\{\phi_K, \phi_f \}^F=X_{\phi_f}(\phi_K)= - \sum_i \phi_{e_i \lrc K} L_p^*(df)(e_i).$$

\


Moreover, by the properties of the Poisson bracket, it follows
$$\{\phi_{K}, \phi_{gL}\}^F=\{\phi_K,\phi_g\}^F\phi_L+\phi_g\{\phi_K,\phi_L\}^F.$$

\


 For any $\xi\in\nn$ we will denote by $\bar\xi$ the right-invariant vector field whose value at $e$ is $\xi$. In other words, 
$\bar{\xi}(p)={R_p}_*(\xi)$, where $R_p$ denotes the translation on the right by $p\in N$. 
The flow of $\bar\xi$ is given by $\varphi^{\bar\xi}_t(p)=R_p (\exp(t\xi))$, i.e. $\varphi^{\bar\xi}_t$ is the left translation by $\exp(t\xi)$.
Since translations on the left by elements of the Lie group are isometries of $(N, \la\,,\,\ra)$, any right-invariant vector field is a Killing vector field.

Note that $\exp(t\xi)p=pp^{-1}\exp(t\xi)p$, so that
 $\left.\frac{d}{dt}\right|_{t=0}\exp(t\xi)p= \left.\frac{d}{dt}\right|_{t=0} L_p \Ad(p^{-1})\exp(t\xi),$
which shows that the value at $p$ of the right-invariant vector field $\bar\xi$determined by $\xi$ is equal to the value at $p$ of the left-invariant vector field determined by $\Ad(p^{-1})\xi$. This was already noticed in \eqref{barv}.

The following result proves the existence of magnetic Killing vector fields induced by right-invariant vector fields, using \ref{prop1}.

\begin{cor}\label{coriv}
	Every right-invariant vector field $\bar \xi$ on the simply connected Lie group $N$ determines a magnetic first integral of the magnetic flow associated to any left-invariant Lorentz force $F$. 
\end{cor}
\begin{proof} Since $\bar\xi$ is Killing, it holds $d^s \bar\xi=0$. By Proposition \ref{prop1} applied for $l=1$, we need to show that there exists $f$ such that $df=F(\bar\xi)^\flat$. 
	
	Note that $d(F(\bar\xi)^\flat)=d(\bar\xi \lrc \omega^F)=\mathcal{L}_{\bar\xi}(\omega^F) - \bar\xi \lrc d \omega^F=0$.
	Indeed, the last term vanishes since $\omega^F$ is closed, and the first term vanishes since $\omega^F$ is left-invariant, and the flow of $\bar\xi$ consists of left translations. Therefore, since $N$ is simply connected and every closed form is exact, there exists a function  $f$ such that $df=F(\bar\xi)^\flat$.
\end{proof}
Later in Proposition \ref{prop4} we shall give the explicit expression of the function $f$ on 2-step-nilpotent Lie groups.

\smallskip

A {\em magnetic Killing vector field} as in the corollary above will be denoted by $\xi^F=\bar\xi+f$. It is uniquely determined by $\xi$ up to a constant.

\begin{lem}\label{lemA} Let $\bar\xi$ denote a right-invariant (Killing) vector field on the manifold $N$ induced by $\xi\in\nn$, with flow denoted by $\varphi_t$ (recall that $\varphi_t(p)=\exp(t\xi)p$ for every $p\in N$). Then the flow of the Hamiltonian vector field $X_{\phi_{\bar\xi}}$ on $\Ta M$ is $(\varphi_t)_*$.
\end{lem}
\begin{proof} Let $Y$ denote the vector field on $\Ta N$ whose flow is given by $\Phi_t:=(\varphi_t)_*$. Since $\pi\circ \Phi_t=\varphi_t \circ \pi_*$, we have $\pi_*Y=\bar \xi$. In particular, this implies that
\begin{equation}\label{etaY}\eta(Y)=\phi_{\bar\xi}.
\end{equation}
	
	For every $X\in \nn$, consider the vector fields  $X^*$ and $\widetilde{X}$ on $\Ta N$ defined  in Equation \eqref{Verthor}.  Let $\eta$ the 1-form given in Equation \eqref{eta}. We claim that: 
	\begin{equation}\label{assert1}\eta([Y,U])=Y(\eta(U))=0 \quad \mbox{ for $U= X^*$ or }\widetilde{X}.
		\end{equation}
	Indeed, 
	\begin{itemize}
		\item if $U=X^*$, then $\pi_*(U)=0$ and $\pi_*(Y)=\bar\xi$, which implies $\pi_*[Y,U]=0$. Therefore $\eta([Y,U])=0$ and on the other hand, since  $\eta(U)=0$, we get $Y(\eta(U))=0$.
		\item If $U=\widetilde{X}$, one has $\pi_*(U)=X$ (or more precisely the left-invariant vector field defined by $X$) and so, $\eta_{pV}(U)=\la X,V\ra$ for every $p\in N$ and $V\in \nn$. For every $t\in \RR$ we can write $\Phi_t(pV)=(\varphi_t)_*(pV)=\varphi_t(p)V$. Thus 
		$$Y(\eta(U))(pV)	 =  \left. \frac{d}{dt}\right|_{t=0}\eta_{\Phi_t(pV)}(U)= \left. \frac{d}{dt}\right|_{t=0}\la X,V\ra=0,$$
	and $\eta_{pV}([Y,U])=\la V, \pi_*[Y,U]\ra =\la V, [\bar\xi,X]
	\ra=0$, since the right-invariant vector field $\bar\xi$ commutes with the left-invariant vector field $X$.
		\end{itemize}
	
	 Then, using the fact that the expression $\eta([Y,U])-Y(\eta(U))=-(\mathcal{L}_Y\eta)(U)$ is tensorial in $U$, we get
	\begin{equation}\label{assert2}\eta([Y,U])=Y(\eta(U)) \qquad \mbox{ for every }U\in\Ta(\Ta N).\end{equation}
	
	From this, together with \eqref{etaY} and \eqref{assert2}, we obtain for every vector field $U\in \chi(\Ta N)$:
\begin{eqnarray*}
		\Omega(Y,U) & = & -d\eta_V(Y,U) = -Y(\eta(U)) + U(\eta(Y)) + \eta([Y,U])\\
		& = & U(\eta(Y))=U(\phi_{\bar\xi})\\
		& = & \Omega(X_{\phi_{\bar\xi}}, U).
	\end{eqnarray*}
		Since $\Omega$ is non-degenerate, this proves that  $Y=X_{\phi_{\bar\xi}}$. 
		\end{proof}
		As an immediate consequence, we remark that the Hamiltonian vector field $X_{\phi_{\bar\xi}}$  satisfies $\pi_*(X_{\phi_{\bar\xi}})=\bar\xi$.
	
%

In the following lemma we establish the relation of $X_{\phi_{\bar\xi}}$ with the Hamiltonian vector field associated with the Lorentz force $F$. 

\begin{lem}\label{lemB}	Let $F$ denote a Lorentz force on $N$. Let $X^F_H$ and $X_H$ denote the Hamiltonian vector fields of $H:\Ta N\to \RR$, defined with respect to the symplectic forms $\Omega^F$ and $\Omega$ on $\Ta N$. Then 
	$$X^F_{\phi_{\xi^{F}}}=X_{\phi_{\bar\xi}},$$
for every right-invariant vector field $\bar\xi$,	where $\xi^F=\bar\xi+f$ was defined in \eqref{magneticKilling}.
	\end{lem}
\begin{proof} Let $\Omega^F$ denote the twisted symplectic form on $\Ta N$. Then, using $df=(F\bar\xi)^\flat$ we compute for every $U\in\Ta(\Ta N)$:
	\begin{eqnarray*}
		\Omega^F(X^F_{\phi_{\xi^F}},U)& = & U(\phi_{\xi^F})=U(\phi_{\bar\xi}+\phi_f)\\
	&= & \Omega(X_{\phi_{\bar\xi}}, U)+ U(\pi^*f)=\Omega (X_{\phi_{\bar\xi}}, U)+ \la F(\bar\xi), \pi_*U\ra\\
	& = & \Omega (X_{\phi_{\bar\xi}}, U) + \omega^F(\pi_* X_{\phi_{\bar\xi}}, \pi_*U)\\
		& = & \Omega (X_{\phi_{\bar\xi}}, U)  + \pi^* \omega^F(X_{\phi_{\bar\xi}}, U) = \Omega^F(X_{\phi_{\bar\xi}}, U).
	\end{eqnarray*}
Since $\Omega^F$ is non-degenerate, the conclusion follows.
\end{proof}

\begin{prop}\label{prop2} Let $\xi^F$ denote a  magnetic Killing vector field (Definition \ref{magneticKilling}) and let $K\in \Gamma(\Sym^k \Ta N)$ be a symmetric tensor. Then 
	$$\{\phi_{\xi^F}, \phi_K\} ^F=-\phi_{\mathcal L_{\bar\xi}K}.$$
\end{prop}
\begin{proof} For every $V\in\Ta(\Ta N)$ we compute:
\begin{eqnarray*}
	\{\phi_{\xi^F}, \phi_K\}^F(V) & = & \Omega_V(X^F_{\phi_{\xi^F}}, X^F_{\phi_K})=-X^F_{\phi_{\xi^F}}(\phi_K)(V)\\
	& = & -X_{\phi_{\bar\xi}}(\phi_K)(V)=-\left. \frac{d}{dt}\right|_{t=0}\phi_K((\varphi_t)_*V)\\
	& = & -\frac{1}{k!}\left. \frac{d}{dt}\right|_{t=0}(\varphi^*_tK)(V, \hdots, V) = -\frac{1}{k!} (\mathcal L_{\bar\xi}K)(V,\hdots, V)\\
	& = & -\phi_{\mathcal L_{\bar\xi}K}(V).
		\end{eqnarray*}
\end{proof}

We end up this section with the following useful criterion:

\begin{thm}\label{prop3} Let $\xi^F_1, \xi_2^F$ denote magnetic Killing vector fields on $N$, written by Definition \ref{magneticKilling} as $\xi_i^F=\bar\xi_i+f_i$, with $\bar\xi_i$ right-invariant determined by $\xi_i\in\nn$, and $df_i=(F\bar\xi_i)^\flat=\bar\xi_i\lrc\,\omega^F$ for $ i=1,2$. Then 
	$$\{\phi_{\xi_1^F}, \phi_{\xi_2^F}\}^F=0 \quad \iff  \quad [\xi_1,\xi_2]=0 \mbox{ and } \omega^F(\xi_1,\xi_2)=0.$$
		\end{thm}
	\begin{proof} Using the fact that $d(\bar\xi_i\lrc\,\omega^F)=\mathcal{L}_{\bar\xi_i}\omega^F=0$, we get:
%
	\begin{eqnarray*}
		d(\omega^F(\bar\xi_1,\bar\xi_2)) & = & d(\bar\xi_2\lrc\,(\bar\xi_1\lrc\,\omega_F)) = \mathcal L_{\bar\xi_2}(\bar\xi_1\lrc\,\omega_F)- \bar\xi_2 \lrc\, d(\bar\xi_1\lrc \omega^F)\\
	& = & ({\mathcal L}_{\bar\xi_2}\bar\xi_1)\lrc\,\omega^F=\overline{[\xi_1,\xi_2]}\lrc\,\omega_F
\end{eqnarray*}
(note that $[\bar\xi_2,\bar\xi_1]=\overline{[\xi_1,\xi_2]}$). Therefore, if $[\xi_1,\xi_2]=0$, then $\omega^F(\bar\xi_1,\bar\xi_2)$ is constant. On the other hand, by Proposition \ref{prop2} we have:
$$
	\{\phi_{\xi_1^F}, \phi_{\xi_2^F}\}^F =  \phi_{\mathcal L_{\bar\xi_2}\xi_1^F}=\phi_{[\bar\xi_2,\bar\xi_1]}(V)+\phi_{\bar\xi_2(f_1)} = \phi_{\overline{[\xi_1,\xi_2]}} + \phi_{\omega^F(\bar\xi_1,\bar\xi_2)},
$$
which proves the assertion.
	\end{proof}

\section{Examples of Liouville Integrable systems on 2-step nilmanifolds}

The goal of this section is to show examples of Liouville integrable magnetic flows on compact manifolds of the form $\Lambda \backslash N$, where $N$ is a 2-step nilpotent Lie group and $\Lambda$ denotes a cocompact lattice in $N$. In this framework, any left-invariant metric on $N$ descends to a Riemannian metric on $\Lambda \backslash N$ since $\Lambda$ acts by isometries. Examples will be the Kodaira-Thurston manifold in dimension 4 and Heisenberg nilmanifolds of any dimension.

We start by studying the integrability of the magnetic flow at the Lie group level.
In order to construct first integrals on the Lie group we make use of the functions associated to symmetric tensors studied in the previous sections, which later will be adapted for the quotient spaces. 

\smallskip

Recall that a Lie group is 2-step nilpotent if its Lie algebra is 2-step nilpotent, that is, $[U,[V,W]]=0$ for all $U,V,W\in \nn$. 

If $N$ is simply connected and 2-step nilpotent, the exponential map is a diffeomorphism $\exp:\nn \to N$, and the Baker–Campbell–Hausdorff formula reads
\begin{equation}\label{bch}\exp(X)\exp(Y)=\exp(X+Y+\frac12 [X,Y]), \quad \forall X,Y\in \nn.
\end{equation}

Let $\la\,,\,\ra$ be a left-invariant metric on $N$, determined at the Lie algebra level by a scalar product on $\nn$. We denote by $\zz$ the center of $\nn$ and consider the orthogonal decomposition:
\begin{equation}\label{decompn}
\nn =\vv \oplus \zz, \quad\mbox{ where } \vv=\zz^{\perp}.
\end{equation}
Furthermore, we consider the skew-symmetric maps $j(Z):\vv \to \vv$ given by 
$$\la j(Z) U,W\ra: =\la Z, [U,W]\ra \qquad \forall U,W\in \vv.$$
Since $\nn$ is 2-step nilpotent, the maps $\{j(Z)\}_{Z\in \zz}$, encode the Lie algebra structure of $\nn$ and the geometrical information of $(N,\la\,,\,\ra)$ see \cite{Eb}.

    Let $F$ be a left-invariant Lorentz force on the 2-step nilpotent Lie group $(N, \la\,,\,\ra)$ and let $\bar\xi$ denote a right-invariant vector field on $N$ (hence Killing) induced by some vector $\xi\in\nn$. In the next paragraphs we shall compute the function $f:N \to \RR$ such that $df=(F\bar\xi)^\flat$. 

For any $X\in\nn$, let $f_X: N\to \RR$ denote the function given by 
\begin{equation}\label{fx}f_X(\exp(W))=\la X, W\ra,\qquad \forall  W\in \nn.
\end{equation}
 We compute below the differential of $f_X$ at $p=\exp(W)\in N$. 
 For every $Y\in\nn$ we have by \eqref{bch}:
\begin{eqnarray*}
df_X(pY) & = & \left.\frac{d}{dt}\right|_{t=0}f_X(\exp(W)\exp(tY))=\left.\frac{d}{dt}\right|_{t=0}\la X,W+tY + \frac12 t[W,Y]\ra \\
& = & \la X,Y \ra  + \frac12 \la X, [W,Y]\ra\\
& = & \la p X, p Y\ra - \frac12 \la X,z\ra \la j(z)Y,W \ra,
\end{eqnarray*}
which implies that $p^{-1}(df_X)_p=X - \frac12 \la X,z\ra \sum f_{j(z)v_i} v_i,$ where  $v_i$ is any orthonormal basis of $\vv$. 

The following proposition is related to Corollary \ref{coriv} for the 2-step nilpotent situation. 

\begin{prop}\label{prop4} Let $(N, \la\,,\,\ra)$ denote a $2$-step nilpotent Lie group equipped with a left-invariant metric and let $F$ denote a left-invariant Lorentz force. The following relation holds: 
	$$d\left(f_{F(\xi)} - \frac12 \sum f_{v_i}f_{v_j} \la F[v_i, \xi], v_j\ra\right) = F(\bar\xi_p)^\flat.$$
	\end{prop}
\begin{proof}
For $p=\exp(W)\in N$ we have $\bar\xi_p=p \Ad(p^{-1})(\xi)=p(\xi-[W,\xi])$ and so, $F(\bar\xi_p)=p \left(F(\xi)-F[W,\xi]\right)$. 
	
Consider the 1-form $\sigma:=d\left( f_{F( \xi)} - \frac12 \sum_i f_{ v_i}f_{ v_j} \la F[ v_i,  \xi],  v_j\ra\right)$. 	From the computations above and using the fact that $\omega^F$ is closed, we obtain:
\begin{eqnarray*}
	p^{-1}\sigma_p & = & F( \xi)-\frac12 \la F( \xi),z\ra \sum_i f_{j(z) v_i} v_i - \frac12 \sum_{i,j}  v_i f_{ v_j}\la F[ v_i, \xi],  v_j\ra - \frac12 \sum_i f_{ v_i} F[ v_i, \xi]\\
	& = & F(\xi) -\frac12 \la F(\xi),z\ra \sum_i f_{j(z) v_i} v_i -\frac12 \sum_{i,j}  v_i f_{ v_j}(\la F[ v_j, \xi],  v_i\ra + \la F[ v_i, v_j],  \xi \ra) \\
	&&-\frac12\sum_i f_{ v_i} F[ v_i, \xi]\\
& = & F( \xi)-\frac12 \la F(\xi),z\ra \sum_i f_{j(z) v_i} v_i - \sum_i f_{ v_i} F[ v_i, \xi] -\frac12 \sum_{i,j}  v_i f_{ v_j} \la F(z), \xi\ra \la j(z) v_i, v_j\ra\\
& = & F( \xi) - \sum_i f_{ v_i} F[ v_i, \xi].
	\end{eqnarray*}
\end{proof}

\begin{example}\label{h3} The Heisenberg Lie group of dimension $3$ can be identified with $\RR^3$ together with the product given by
	$$(x_1,x_2,x_3)(y_1,y_2,y_3) = (x_1+y_1,x_2+y_2,x_3+y_3+\frac12 (x_1y_2-x_2y_1)).$$
	Usual computations show that a basis of left-invariant vector fields at $p=(x,y,z)$ is given by
	$$e_1(p)=\partial_x -\frac12y \partial_z, \qquad e_2(p)=\partial_y + \frac12 x \partial_z, \qquad e_3(p)=\partial_z,$$
	where we denote $\partial_u=\frac{\partial}{\partial u}$. These vector fields obey the only non-trivial Lie bracket relation
	$[e_1,e_2]=e_3$.
	This Lie algebra is usually denoted by $\hh_3$. 
	Take the metric that makes of $\{e_1, e_2, e_3\}$ an orthonormal basis. Then $\zz=span\{e_3\}$, while $\vv=\zz^{\perp}=span\{e_1,e_2\}$. The set $\{j(Z)\}_{Z\in \zz}$ as above is generated by $J:j(e_3)$ whose matrix in the basis $\{e_1,e_2\}$ is
	$$\left( \begin{matrix}
		0 & -1\\ 1 & 0
	\end{matrix} \right).
	$$

	Let $\nn:=\hh_3\oplus\mathbb{R}e_4$ denote the trivial extension of  $\hh_3$. This is also a 2-step nilpotent Lie algebra with Lie group underlying $\RR^4$ and multiplication given by
	$$(x_1,x_2,x_3,x_4)(y_1,y_2,y_3,y_4) = (x_1+y_1,x_2+y_2,x_3+y_3+\frac12 (x_1y_2-x_2y_1),x_4+y_4).$$
	
We extend the metric of $\hh_3$ to $\nn$ in such way that $e_4\perp \hh_3$ and $\la e_4,e_4\ra=1$. 	In this case the Lie algebra decomposes as
 orthogonal direct sum	
 $$
 \nn=\vv\oplus C(\nn) \oplus \aa,
 $$
	where $\vv$ as above,  $C(\nn)=span\{e_3\}$ is the commutator of the Lie algebra and $\aa=span\{e_4\}$. In this case  $j(e_3)=J$ and $j(e_4)=0$. 
	\end{example}

We shall now study the Liouville integrability of the magnetic flow. In general one searches for analytical first integrals but this could be not possible (see for instance \cite{Ta} for geodesic fields).


\begin{defn}\label{completeintegrability} The magnetic flow  on a Riemannian manifold $(M, \la\,,\,\ra)$ of dimension $n$ for the Lorentz force $F$, is {\em completely integrable} or {\em Liouville integrable} 
if there exists a family of $n-1$ first integrals $\Psi_j:\Ta M \to \RR$, $j=1, \hdots, n-1$
such that
	\begin{itemize}
		\item $\{\Psi_i,\Psi_j \}=0$ for all $i,j$ and 
		\item  The gradients of  $	\{\En, \Psi_1, \hdots, \Psi_{n-1} \}$ are linearly independent on a dense subset  of $\Ta M$.
	\end{itemize}
\end{defn}


\smallskip

As an application of the framework developed above, we will now prove complete integrability of the magnetic flow on some 2-step nilpotent Lie groups equipped with a left-invariant metric and a left-invariant Lorentz force.

\begin{prop}\label{dim4} Let $\nn$ denote a $2$-step nilpotent Lie algebra of dimension $4$ equipped with a scalar product. If the Lorentz force $F$ is left-invariant and has rank $2$, then the corresponding magnetic flow on the simply connected Lie group $N$ is completely integrable. 
\end{prop}
\begin{proof}
	By Theorem \ref{prop3} it suffices to find three right-invariant vector fields $\bar\xi_1, \bar\xi_2, \bar\xi_3$ on $N$, corresponding to linearly independent vectors $\xi_1, \xi_2, \xi_3\in \nn$, such that 
	\begin{equation}\label{eij}[\xi_i,\xi_j]=0,\quad\mbox{ and }\quad\omega^F(\xi_i,\xi_j)=0,\qquad \forall i,j.
	\end{equation}

In view of the decomposition \eqref{decompn} for the Lie algebra $\nn$, we have $\zz=C(\nn) \oplus \aa$, with  $j:C(\nn)\to \vv$ injective   $\Rightarrow\, \dim \vv\geq 2$. If $\dim C(\nn)=2$, then $\dim(v)\leq 2$ and $j$ is not injective. Thus, $\dim C(\nn)=1$, $\dim\vv=2$ and $\dim\aa=1$.  

There is a orthonormal basis, $\{e_1, e_2, e_3, e_4\}$ of $\nn$ such that $\vv=span\{e_1, e_2\}$, $C(\nn)=span\{e_3\}$, and $\aa=span\{e_4\}$. Thus, 
we have $[e_1,e_2]=\lambda e_3$ for some $\lambda\neq 0$. 

Since the magnetic field is closed, then $\omega^F(e_3,e_4)=0$. By using a rotation in $\vv$, (that is, by application of an isometry, see \cite{OS3}), we may assume that $\omega^F(e_2,e_4)=0$. In this case, $\omega^F$ has the form
\begin{equation}\label{of}\omega^F=a e^{13}+ b e^{23}+ c e^{14}+ \rho e^{12}.\end{equation}
Here we denote by $e^{rs}$ the 2-form $e^r\wedge e^s$, where $\{e^i\}$ is the dual basis of $\{e_i\}$ on $\nn$. 

The assumption that $F$ has rank 2 is equivalent to $\omega^F\wedge\omega^F=0$, that is $bc=0$. We distinguish the following two cases:

\begin{itemize}
\item If $c=0$ and $b\ne 0$, we consider the vectors $\xi_1:=be_1-ae_2,\ \xi_2:=e_3,\ \xi_3:=e_4$.

\item  If $b=0$, we consider the vectors $\xi_1:=e_2,\ \xi_2:=e_3,\ \xi_3:=e_4$. 

\end{itemize}	
	It is not hard to check that they satisfy the required conditions in Equation \eqref{eij} in each case. 
\end{proof}

We shall now prove the complete integrability of the magnetic flow for a family of Lorentz forces on some extensions to higher dimensions of the Heisenberg Lie group of Example \ref{h3}.

\begin{example} The Heisenberg Lie group of dimension $2n+1$ has the underlying differentiable structure of $\RR^{2n} \times \RR$ together with the product given by
		$$(v_1,z_1)(v_2,z_2)=(v_1+v_2, z_1+z_2-\frac12 v_1^t J v_2),\qquad \forall v_1,v_2\in \RR^{2n}, \ \forall z_1,z_2\in \RR,$$
		where $J$ is the standard complex structure in $\RR^2n$ defined by $J(e_{2i-1})=e_{2i}$ and $J(e_{2i})=-e_{2i-1}$.
		For $v=(x_1, \hdots, x_{2n})$, a basis of left-invariant vector fields is given by
		$$e_{2i-1}(p)=\partial_{2i-1} - \frac12 x_{2i}\partial_z\quad e_{2i}(p)=\partial_{2i} + \frac12 x_{2i-1}\partial_z \quad e_{2n+1}=\partial_z, \quad \mbox{ for }i=1, \hdots, n. $$
		They satisfy the non-trivial Lie bracket relations
		$$[e_{2i-1}, e_{2i}]=e_{2n+1},  \qquad \mbox{ for }i=1, \hdots, n. $$
We denote the corresponding Lie algebra by $\hh_{2n+1}$. Choose the metric that makes the basis above an orthonormal basis. Note that we have $C(\hh_{2n+1})=\RR e_{2n+1}$, and the maps $j(Z)$ are spanned by $j(e_{2n+1}):=J:\vv \to \vv$ for $\vv=C(\hh_{2n+1})^{\perp}=\RR^{2n}$.

In particular a skew-symmetric  derivation $D$ on $\hh_{2n+1}$ may satisfy  $JD=DJ$. In fact, since $D[U,V]= [DU,V]+ [U,DV]$ we derive that for any  $Z\in \zz$ it holds $DZ=0$ and this implies that $D\vv\subseteq \vv$ and $[DU,V]+[U,DV]=0$ for all $U,V\in \vv$. In particular one has $\la DU, [V,W]\ra =0$ for all $U,V,W\in \vv$.
	\end{example}

\begin{rem} Skew-symmetric  derivations of $\hh_{2n+1}$ give rise to Lorentz forces of type I, according to the terminology of \cite{OS}, while a magnetic field on a 2-step nilpotent Lie algebra of dimension 4 is sum of a magnetic field of type I and type II. Magnetic fields of type II do not exist in Heisenberg Lie algebras $\hh_{2n+1}$ with $n\leq 2$, see \cite{OS}.
\end{rem}

\begin{prop}\label{integhn}
	Let $\nn=\hh_{2n+1}\oplus \RR^k$ denote a trivial extension of the Heisenberg Lie algebra of dimension $2n+1$ equipped with a product metric such that $\hh_{2n+1}\perp\RR^k$. Let $F$ be a skew-symmetric derivation of $\nn$ such that $\left.F\right|_{\RR^k}\equiv 0$. Then the corresponding magnetic flow is completely integrable on  the associated simply connected Lie group.
\end{prop}
\begin{proof}
	Since $F$ is a skew-symmetric derivation on $\hh_{2n+1}$, it vanishes on the center $\RR e_{2n+1}$ and it satisfies $JF=FJ$ on the subspace $\vv\subset \hh_{2n+1}$. Consequently, there exists an orthonormal basis $v_1, \hdots, v_{2n}$ of $\vv$ such that
	\begin{equation}\label{JF}
	    J = \sum_{i=1}^n  v_{2i-1}\wedge v_{2i}, \quad F = \sum_{i=1}^n \eta_i v_{2i-1}\wedge v_{2i}, \quad \mbox{ for } i=1, \hdots, n.\end{equation}
	Note that $\eta_i\in \RR$ are not all zero.  
	Let $Z_1, \hdots, Z_k$ denote a orthonormal basis of left- (and right-)invariant vector fields on $\RR^k$. 
	
	Let $\bar v_j$ denote as before the right-invariant vector field defined by $v_j$. We denote $v_{2n+1}:=e_{2n+1}$. By Corollary \ref{coriv} and Theorem \ref{prop3},  $\{\phi_{v_1^F}, \phi_{v_3^F}, \hdots, \phi_{v_{2n+1}^F}\}$ is  a family of $n+1$ commuting first integrals of the magnetic flow. 
	
	Let $S_i$ denote the left-invariant symmetric tensor on $N$ whose value at the identity is $S_i=v_{2i-1}^2+v_{2i}^2$ for $i=1, \hdots, n$. Denote by $\phi_{S_i}$ the corresponding function on $\Ta N$. 
	
	{\it Assertion.} The set of functions on $\Ta N$, $\{\phi_{v_1^F}, \phi_{v_3^F}, \hdots, \phi_{v_{2n+1}^F}, \phi_{Z_1}, \hdots, \phi_{Z_k}, \phi_{S_1}, \hdots,\phi_{S_n}\}$ gives a family of first integrals in involution for the Poisson bracket associated to the Lorentz force $F$. 
	
	This follows from Theorem \ref{prop3} and Proposition \ref{prop2}.
\end{proof}

\section{Complete integrability on compact spaces}

Let $\Lambda$ denote a cocompact lattice on a 2-step nilpotent Lie group $N$, that is, $\Lambda$ is a discrete subgroup of $N$ such that $M:=\Gamma\backslash N$ is a compact space. 

Elements of $\Lambda \backslash N$ are denoted by $\bar{p}$ and represent the class $\{\gamma p\ ,\ \gamma\in \Gamma\}$.

Since the metric on $N$ is invariant by translations on the left by elements of $N$, it descends to $\Gamma \backslash N$ in such a way that $\pi:N \to \Gamma \backslash N$ is a local isometry. 

In this way we also define the corresponding Lorentz force $F$ on the quotient space $M$ and consider the corresponding magnetic flow on $M$. The aim now is to construct first integrals for this magnetic flow on $\Gamma \backslash N$. 

Let $\gamma\in \Gamma$, that we write as $\gamma=\exp(\lambda)$. We compute $\gamma^*(f_X)$ for $X\in \nn$, where $f_X$ was defined in \eqref{fx}:
\begin{eqnarray*}
	\gamma^*(f_X)(\exp(W))& = & f_X(\exp(\lambda)\exp(W))=\la X, \lambda + W + \frac12 [\lambda,W]\ra\\
	& = & \la X, \lambda\ra + f_X(\exp(W))+\frac12 \sum_k\la X,z_k\ra f_{j(z_k)\lambda}(\exp(W)),
\end{eqnarray*}
where $\{z_k\}$ is an orthonormal basis of the center $\zz$ of $\nn$. 
Consequently, we get
\begin{equation}\label{gfx}\gamma^* f_X=f_X+\la X, \lambda \ra +\frac 12 \sum_i \la X,z_i\ra f_{j(z_i)\lambda}.
\end{equation}

Consider now some $\xi\in\nn$ and the magnetic Killing vector field $\xi^F=\bar\xi + f$ given by Corollary \ref{coriv}. By Proposition \ref{prop4}, the associated first integral is given by
\begin{equation}\label{phixi}
\phi_{\xi^F}(pV)=\la pV,\xi p \ra + f_{F(\xi)}(p) - \frac12 \sum f_{ v_i}(p) f_{ v_j}(p)\la F[ v_i,\xi], v_j\ra,
\end{equation}
where $\{v_i\}$ is an orthonormal basis of $\vv$. 
To compute $\gamma^* \phi_{\xi}$, set $\gamma=\exp(\lambda)$ as above. By using \eqref{gfx} we obtain:
\begin{eqnarray*}
\phi_{\xi^F}(\gamma pV) & = & \la \gamma p V, \xi \gamma p\ra + f_{F(\xi)}(\gamma p)-\frac12 \sum_{i,j} f_{ v_i}(\gamma p)f_{ v_j}(\gamma p)\la F[ v_i,\xi], v_j\ra\\
& = & \la pV, \gamma^{-1}\xi\gamma p\ra + (f_{F(\xi)}(p)+\la F(\xi), \lambda\ra + \frac12\sum_k  \la F(\xi), z_k\ra f_{j(z_k)\lambda}(p))\\
& & - \frac12 \sum_{i,j} (f_{ v_i}(p)+\la  v_i,\lambda\ra)(f_{ v_j}(p)+\la  v_j,\lambda\ra)\la F[ v_i,\xi], v_j\ra.	
\end{eqnarray*}

Recall that we have $Ad(\gamma^{-1})\xi = \xi-[\lambda, \xi]$. Thus, using the fact that $\omega^F$ is closed, together with the computation above gives:
\begin{eqnarray*}
	\phi_{\xi^F}(\gamma pV) & = & I_{\xi^F}(pV) - \la pV, [\lambda,\xi]p\ra +\la F(\xi),\lambda \ra + \frac12\sum_k  \la F(\xi),z_k\ra f_{j(z_k)\lambda}(p) \\
& &  - \frac12\left( \sum_i f_{ v_i}(p)\la F[ v_i,\xi],\lambda\ra  +\sum_j f_{ v_j}(p)
\la F[\lambda,\xi], v_j\ra +\la F[\lambda,\xi],\lambda\ra\right)\\
& = & I_{\xi^F}(pV)- \la pV,[\lambda,\xi]p\ra + \la F(\xi),\lambda \ra  +\frac12\sum_{i,k}  \la F(\xi),z_k\ra f_{v_i}(p)\la z_k,[\lambda,v_i]\ra\\
& &
-\frac12\left(2\sum_i f_{ v_i}(p) (\la F[\lambda,\xi], v_i\ra+ \la F[ \lambda,v_i],\xi\ra) + \la F[\lambda,\xi],\lambda\ra \right)\\
& = & \phi_{\xi^F}(pV) - I_{[\lambda,\xi]^F}(pV)+ \la F(\xi),\lambda\ra -\frac12 \la F[\lambda,\xi],\lambda\ra\\
& = & \phi_{(\xi-[\lambda,\xi])^F}(pV)+\la F(\xi),\lambda\ra -\frac12 \la F[\lambda,\xi],\lambda\ra.
\end{eqnarray*}
Notice that the expression $\la F(\xi),\lambda\ra -\frac12 \la F[\lambda,\xi],\lambda\ra$ is a constant function. 

Therefore for every $\xi\in \nn$ and $\gamma\in\Gamma$, the pull-back of the associated first integral of the magnetic flow by the left action of $\gamma$ on $\Ta N$, is given by
\begin{equation}\label{lambdamagnetic}
	\gamma^* \phi_{\xi^F}=\phi_{(\xi-[\lambda,\xi])^F}+\la F(\xi)-\frac12 F[\lambda,\xi],\lambda\ra.
\end{equation}

The aim now is to adapt the first integrals of Proposition \ref{dim4} to a compact quotient $\Gamma \backslash N$. 

\subsection{Magnetic flows on 4-dimensional nilmanifolds} Consider the group $N=H_3\times \RR$ endowed with a left-invariant metric $g$ and a left-invariant force $F$ of rank $2$ with corresponding closed 2-form $\omega^F$. For every co-compact lattice $\Gamma\subset N$, we consider the compact manifold $M_\Gamma:=\Gamma\backslash N$, with the induced Riemannian metric and closed 2-form, also denoted by $g$ and $\omega^F$. 

There exists an orthonormal basis $\{e_1,e_2,e_3,e_4\}$ of $\nn$ whose only non-trivial commutator relation is $[e_1,e_2]=e_3$.
The Lie algebra of $\nn$ can be identified with $\RR^4$ via this basis. For every non-zero real numbers $x_1,\ldots,x_4$ we will consider the lattice $\Lambda$ of $\nn$ spanned by $\{x_1e_1,\ldots,x_4e_4\}$. 
Let $\Gamma:=\exp(\Lambda)$ be the image by the exponential map of the lattice $\Lambda$. The set $\Gamma$ is clearly discrete, and by \eqref{bch}, it is a subgroup of $N$ if and only if 
\begin{equation}\label{lat}
    \frac{x_1x_2}{2x_3}\in\mathbb{Z}. 
\end{equation}
Thus $\Gamma$ a co-compact lattice of $N$ if and only if \eqref{lat} holds. Note that in this case, $M_\Gamma$ is diffeomorphic to the Kodaira-Thurston manifold $\mathrm{KT}^4$.

We will adapt the first integrals of the magnetic flow on $N$ given by Proposition \ref{dim4} (with $\omega^F$ given by \eqref{of}), in order to obtain first integrals of the corresponding magnetic flow on $\Gamma \backslash N$ with the induced force for some specific choices of $\Gamma$. For this, we need to construct first integrals on $N$ invariant by $\Gamma$. 

We introduce the simpler notation $I_j:=\phi_{{e_j}^F}$ for the first integrals \eqref{phixi} associated to the right-invariant vector fields determined by $e_j$ for $j=1,2,3,4$.

Let $\gamma=\exp(\lambda)$ be any element of $N$, acting on $N$ by left translation. Using \eqref{of} and \eqref{lambdamagnetic} one gets
\begin{equation}\label{gi}
\begin{array}{rcl}
	\gamma^*I_1 & = &\phi_{(e_1-[\lambda,e_1])^F}+\la F(e_1)-\frac12  F[\lambda,e_1], \lambda \ra\\
    \gamma^*I_2 & = &\phi_{(e_2-[\lambda,e_2])^F}+\la F(e_2)-\frac12  F[\lambda,e_2], \lambda \ra\\
	\gamma^* I_3 & = & \phi_{e_3^F}+\la F(e_3), \lambda\ra=I_3- \la ae_1+be_2, \lambda\ra\\
	\gamma^* I_4 & = & \phi_{e_4^F}+\la F(e_4), \lambda\ra=I_4- c\la e_1, \lambda\ra.	
\end{array}\end{equation}

We consider two cases, as in the proof of Proposition \ref{dim4}, according to whether $b=0$ or $c=0$ and $b\ne 0$.

$\bullet$ Case $b=0$.
In this case, the commuting first integrals of the magnetic flow on $N$ constructed in Proposition \ref{dim4} are $I_2,I_3,I_4$. 
The system \eqref{gi} implies that if $\lambda=\lambda_1e_1+\lambda_2e_2+\lambda_3e_3+\lambda_4e_4$, then
\begin{equation}\label{pb1}
\begin{split}
      \gamma^* I_2&= I_2-\lambda_1I_3+\la (-\rho) e_1 +\frac{a\lambda_1}2 e_1, \lambda\ra = I_2-\lambda_1 I_3-\rho \lambda_1 +\frac{\lambda_1^2}2 a\\
      \gamma^* I_3&= I_3 - a \lambda_1\\
     \gamma^* I_4&= I_4 - c \lambda_1
\end{split}
 \end{equation}
	
We define 
\begin{equation}\label{l1}
    \Lambda:=\mathrm{Span}_\ZZ(2e_1,e_2,e_3,e_4),
\end{equation} 
which satisfies \eqref{lat}, whence $\Gamma:=\exp(\Lambda)$ is a lattice of $N$. Using \eqref{pb1} it is then straightforward to check that the functions $f_i:\Ta M \to \RR$ defined by
\begin{itemize}
\item $f_2 := \left\{\begin{array}{ll}
	I_2-\frac1{2a}(I_3^2 + \rho)^2 & \mbox{ if } a\neq 0\\
	e^{-1/({I_3+\rho})^2}\sin\left(\frac{\pi I_2}{I_3+\rho}\right)  & \mbox{ if } a=0
\end{array} \right.$
	\item $f_3 := \left\{\begin{array}{ll}
		\sin( \frac{\pi}{a}I_3) & \mbox{ if } a\neq 0\\
		I_3 & \mbox{ if } a=0
		\end{array} \right.$
	\item $f_4 := \left\{\begin{array}{ll}
		\sin(\frac{\pi}{c}I_4 ) & \mbox{ if } c\neq 0\\
		I_4 & \mbox{ if } c=0
	\end{array} \right.$
\end{itemize}
are $\Gamma$-invariant, functionally independent and thus define first integrals in involution of the magnetic flow on $M_\Gamma$.

Note that $f_3$ and $f_4$ are always analytical, and $f_2$ is analytical whenever $a\neq 0$. If $a=0$, then $F(e_3)=0$, hence by \eqref{phixi}, $I_3(pV)=\langle pV,e_3p\rangle=\langle V,e_3\rangle$, for every $p\in N$ and $V\in \nn$, thus showing that $f_2$ is analytical at every energy level $\En<|\rho|$. At energy levels $\En\geq|\rho|$, $f_2$ is only differentiable.

$\bullet$ Case $c=0$, $b\ne 0$. In this case, the first integrals of the magnetic flow on $N$ constructed in Proposition \ref{dim4} are $bI_1-aI_2$, $I_3$ and $I_4$. 

For every $\lambda=\lambda_1e_1+\lambda_2e_2+\lambda_3e_3+\lambda_4e_4\in \nn$, we denote by $\gamma:=\exp(\lambda)$. Then \eqref{gi} shows that the pull-backs of the above first integrals of the magnetic flow on $N$ by the left action of $\gamma$ are given by
\begin{equation}\label{pb}
\begin{split}
      \gamma^* (bI_1-aI_2)&=bI_1-aI_2+(I_3+\rho)(a\lambda_1+b\lambda_2)-\frac12(a\lambda_1+b\lambda_2)^2\\
      \gamma^* I_3&= I_3 - (a \lambda_1+b\lambda_2)\\
      \gamma^* I_4&= I_4
\end{split}
 \end{equation}

We consider the following lattice $\Lambda$ of $\nn$:
\begin{equation}\label{l2}
   \Lambda=\begin{cases}\mathrm{Span}_\ZZ(2e_1,\frac{1}{b}e_2,\frac{1}{b}e_3,e_4)\quad \mbox{if} \ a=0\\
    \mathrm{Span}_\ZZ(\frac{2}{a}e_1,\frac{1}{b}e_2,\frac{1}{ab}e_3,e_4)\quad \mbox{if} \ a\neq0.
\end{cases} 
\end{equation}

By \eqref{lat}, $\Gamma:=\exp(\Lambda)$ is a lattice in $N$, and by \eqref{pb}, the functions $g_i:\Ta M \to \RR$ defined by
\begin{itemize}
\item $g_2 := 	bI_1-aI_2+\frac{1}{2}(I_3 + \rho)^2$
	\item $g_3 := 
		\sin( 2\pi I_3)$
	\item $g_4 := 
		I_4.$
\end{itemize}
are easily seen to be $\Gamma$-invariant, and thus define commuting analytical first integrals of the magnetic flow on $M_\Gamma$.
Note that in the case where $a=b=c=0$ the magnetic 2-form $\omega^F=\rho e^1 \wedge e^2$ is an exact 2-form on the quotient space. Summarizing the above discussion, we have proved the following result.

\begin{thm} Let $\omega^F=a e^{13}+ b e^{23}+ c e^{14}+ \rho e^{12}$ be any invariant closed $2$-form on $\nn:=\hh_3\oplus\RR$, and assume that $\omega^F$ has rank $2$ (so that $bc=0$). Consider the lattice $\Gamma=\exp(\Lambda)$ of $N:=H_3\times \RR$. Then consider the magnetic flow on $\Gamma\backslash N$

\begin{enumerate}
    \item For $\Lambda$  given by \eqref{l1} if $b=0$, the magnetic flow  is completely integrable by analytical first integrals $f_2,f_3,f_4$ either if $a\neq 0$ or if $a=0$ at any  energy level $\En<|\rho|$.
     If $a=0$, one has complete integrability by differentiable (but not all analytical) first integrals at energy levels $\En\geq |\rho|$.
    \item For $\Lambda$  given by \eqref{l2} if $c=0$ and $b\ne 0$, the magnetic flow is completely integrable by analytical first integrals $g_2, g_3,g_4$ at any energy level.
\end{enumerate}

	In particular, this is true for the magnetic field represented by the exact $2$-form $\rho\, d e^3$. 
	\end{thm}

\smallskip

\subsection{Magnetic flows on Heisenberg nilmanifolds of any dimension}
In this section we shall prove integrability of magnetic fields on compact nilmanifolds associated to the  simply connected 2-step nilpotent Lie group $N:=H_{2n+1}\times\RR^k$. 

With the notation from Proposition \ref{integhn}, the set $$\{\phi_{v_1^F}, \phi_{v_3^F}, \hdots, \phi_{v_{2n+1}^F}, \phi_{Z_1}, \hdots, \phi_{Z_k}, \phi_{S_1}, \hdots,\phi_{S_n}\}$$ is a family of first integrals in involution for the Poisson bracket associated to the Lorentz force $F$ given by \eqref{JF}. Since the tensors $Z_i$ and $S_j$ are left-invariant, the corresponding first integrals are automatically $\Gamma$-invariant. It remains to study $\phi_{v_{2j-1}^F}$ for $j=1,\ldots,n+1$. We introduce the notation $I_j:=\phi_{v_{2j-1}^F}$. By \eqref{lambdamagnetic}, we get for every $\gamma:=\exp(\sum_{i=1}^{2n+1}\lambda_iv_i)$:
\begin{eqnarray*}
    \gamma^*I_j&=&I_j+\lambda_{2j}(I_{n+1}+\eta_j),\qquad\mbox{for}\ j\le n,\\
    \gamma^*I_{n+1}&=&I_{n+1}.
\end{eqnarray*}
We define 
\begin{eqnarray*}
     f_j&:=& e^{-1/({I_{2n+1}+\eta_j})^2}\sin\left( \frac{\pi I_j}{I_{2n+1}+\eta_j}\right),\qquad\mbox{for}\ 1\leq j\leq n,\\
     f_{n+1}&:=&I_{n+1}.
\end{eqnarray*}
Since $F(e_{2n+1})=0$, \eqref{phixi} gives $I_{2n+1}(pV)=\langle V,e_{2n+1}\rangle$ for every $p\in N$ and $V\in\nn$. Therefore, $f_j$ is analytical at every energy level $\En<|\eta_j|$, and only differentiable at energy levels $\En\geq|\eta_j|$.

  By taking the lattice $\Lambda:=2\ZZ^{2n+1}\oplus\ZZ^k\subset \hh_{2n+1}\oplus\RR^k$, we obtain by \eqref{bch} that $\Gamma:=\exp(\Lambda)$ is a lattice in $N$ and $f_j$ are $\Gamma$-invariant for $1\leq j\leq n+1$.

The set of first integrals 
$\{f_1,\ldots,f_{n+1}, \phi_{Z_1}, \hdots, \phi_{Z_k}, \phi_{S_1}, \hdots,\phi_{S_n}\}$ 
projects to a family of first integrals in involution for the Poisson bracket associated to the Lorentz force $F$ given by \eqref{JF} on the compact manifold $M:=\Gamma\backslash N$.

\begin{thm}
The magnetic flow for the invariant Lorentz force $F$ on $\Ta M$ with $M=\Gamma \backslash N$, $N$ as above, is completely integrable with analytical first integrals at every energy level $\En<\min(|\eta_1|,\ldots,|\eta_n|)$ and 
  with differentiable (but not all analytical) first integrals at energy levels $\En\geq\min(|\eta_1|,\ldots,|\eta_n|)$. 
  \end{thm}

\noindent{\bf Conflict of interest.} No potential conflict of interest is reported by the authors.

\noindent{\bf Data availability statement.}
Data sharing is not applicable -- the paper describes entirely theoretical research.


\begin{thebibliography}{GGGG}

\bibitem{ABM} {\sc S.V. Agapov, M. Bialy, A.E. Mironov}, {\em Integrable Magnetic Geodesic Flows on 2-Torus: New Examples via Quasi-Linear System of PDEs}, Comm.  Math. Physics, {\bf 351} (3), 993--1007 (2017).

\bibitem{AV} {\sc S.V. Agapov, A. A. Valyuzhenic},  {\em  Polynomial integrals of magnetic geodesic
flows on the 2-torus on several energy levels},  Discrete Contin. Dyn. Syst.,
{\bf 39}(11), 6565--6583 (2019).

 \bibitem{BKM} {\sc A. Bolsinov, A. Konaev, V. Matveev} {\em Integrability of the magnetic geodesic flow on
the sphere with a constant 2-form}, arXiv:2506.23312 (2025).

 \bibitem{Bu} {\sc L. Bulter}, {\em Integrable geodesic flows with wild first integrals: the case of 2-step nilmanifolds}Ergodic Theory Dyn. Systems {\bf 23} (3), 771--797 (2003).

\bibitem{CLMS} {\sc M. Ceki\'c, T. Lefeuvre, A. Moroianu, U. Semmelmann}, {\em Correspondence between Pestov and Weitzenböck identities},
Math. Proc. Cambridge Phil. Soc. {\bf 178} (3), 443--463 (2025).

\bibitem{dBM20} {\sc V. del Barco, A. Moroianu}, {\em Symmetric Killing tensors on nilmanifolds}, Bull. Soc. Math. France {\bf 148} (3), 411--438 (2020).

\bibitem{BdBM25} {\sc R. Berto-Cuevas, V. del Barco, A. Moroianu}, {\em Symmetric Killing tensors on almost abelian Lie groups}, arXiv:2509.26195 (2025).

\bibitem{DGJ1} {\sc V. Dragovi\'c, B. Gaji\'c,  B. Jovanovi\'c}, {\em Integrability of homogeneous exact magnetic flows on spheres}, arXiv:2504.20515 (2025). 

\bibitem{DGJ2} {\sc V. Dragovi\'c, B. Gaji\'c,  B. Jovanovi\'c}, {\em A Lax representation and integrability of homogeneous exact magnetic flows on spheres in all dimensions}, arXiv:2506.23299 (2025).

\bibitem{Eb} {\sc P. Eberlein}, {\em Geometry of 2-step nilpotent Lie groups with a left-invariant metric}, Ann. Sci. ENS., 4 serie {\bf 27} (5),  611--660 (1994). 

\bibitem{HMS} {\sc K. Heil, A. Moroianu, U. Semmelmann}, {\em Killing and Conformal Killing tensors}, J. Geom. Phys. {\bf 106}, 383--400 (2016).

\bibitem{OS} {\sc G. Ovando, M. Subils},  {\em Magnetic fields on non-singular 2-step nilpotent Lie groups}, 
J. Pure Appl. Algebra {\bf 228}, No. 6, Article ID 107618, 24 p. (2024).

\bibitem{OS3} {\sc G. Ovando, M. Subils}, {\em Magnetic trajectories on Heisenberg nilmanifolds}, arXiv:2407.05515 (2024).

\bibitem{Ta} {\sc I. Taimanov}, {\em Topological obstructions to integrability of geodesic flows on non-simply-connected
manifolds}, Mathematics of the USSR-Izvestiya {\bf 30} (2), 403--409 (1988).


\end{thebibliography}
\end{document}